\documentclass[12pt,a4paper,english,reqno]{amsart}
\usepackage[a4paper,footskip=1.5em]{geometry}
\usepackage{amsmath,amssymb,amsthm,mathtools}
\usepackage[mathscr]{euscript}
\usepackage[usenames,dvipsnames]{color}
\usepackage{adjustbox,tikz,calc,graphics,babel,standalone}
\usepackage{subcaption}
\usepackage{csquotes,enumerate,verbatim}
\usepackage[final]{microtype}
\usepackage[numbers]{natbib}
\usetikzlibrary{shapes.misc,calc,intersections,patterns,decorations.pathreplacing}
\usepackage{hyperref}
\hypersetup{colorlinks=true,linkcolor=blue,citecolor=blue,pdfpagemode=UseNone,pdfstartview={XYZ null null 1.00}}
\usepackage{cmtiup}

\theoremstyle{plain}
\newtheorem*{theorem*}{Theorem}
\newtheorem{theorem}{Theorem}[section]
\newtheorem{lemma}[theorem]{Lemma}

\newtheorem*{claim*}{Claim}

\theoremstyle{remark}

\def\N{\mathbb{N}}
\def\Z{\mathbb{Z}}

\DeclareMathOperator\sgn{sgn}

\let\originalleft\left
\let\originalright\right
\renewcommand{\left}{\mathopen{}\mathclose\bgroup\originalleft}
\renewcommand{\right}{\aftergroup\egroup\originalright}

\makeatletter
\def\imod#1{\allowbreak\mkern10mu({\operator@font mod}\,\,#1)}
\makeatother

\begin{document}

\title{Diffusion on graphs is eventually periodic}

\author{Jason Long}
\address{Department of Pure Mathematics and Mathematical Statistics, University of Cambridge, Wilberforce Road, Cambridge CB3\thinspace0WB, UK}
\email{jl694@cam.ac.uk}

\author{Bhargav Narayanan}
\address{Department of Pure Mathematics and Mathematical Statistics, University of Cambridge, Wilberforce Road, Cambridge CB3\thinspace0WB, UK}
\email{b.p.narayanan@dpmms.cam.ac.uk}

\date{7 April 2017}

\begin{abstract}
We study a variant of the chip-firing game called \emph{diffusion}. In diffusion on a graph, each vertex of the graph is initially labelled with an integer interpreted as the number of chips at that vertex, and at each subsequent step, each vertex simultaneously fires one chip to each of its neighbours with fewer chips. Since this firing rule may result in negative labels, diffusion, unlike the parallel chip-firing game, is not obviously periodic. In 2016, Duffy, Lidbetter, Messinger and Nowakowski nevertheless conjectured that diffusion is always eventually periodic, and moreover, that the process eventually has period either 1 or 2. Here, we establish this conjecture.
\end{abstract}

\maketitle

\section{Introduction}
In this paper, we will be be concerned with `chip-firing' games. Given a graph $G$ with piles of chips at each vertex, in the traditional \emph{chip-firing} game, one plays by repeatedly choosing a vertex that has at least as many chips as its degree, and then `firing' this vertex by moving a chip from the vertex to each of its neighbours. This one-player game was introduced by Bj\"orner, Lov\'asz and Shor~\citep{bls}, and the study of dynamics of the chip-firing game and its variants has since grown rapidly, due both to its inherent appeal and the many connections to other areas of mathematics; see~\citep{rec1, rec2, rec3, rec4, rec5} for some examples of recent developments, and the survey of Merino~\citep{survey} for more background.

Here, we will primarily be interested in a variant of the traditional chip-firing game introduced by Duffy, Lidbetter, Messinger and Nowakowski~\citep{diff} called \emph{diffusion}. In diffusion on a finite graph $G$, each vertex of $G$ is initially labelled with an integer interpreted as the number of chips at that vertex, and at each subsequent step, each vertex simultaneously fires one chip to each of its neighbours with fewer chips. In contrast to the parallel chip-firing game~\citep{paral} where every vertex that has at least as many chips as its degree simultaneously fires a chip to each of its neighbours, note that the firing rule in diffusion may result in negative labels even when the initial labels are all positive integers. It is therefore not clear a-priori if diffusion is bounded, and consequently, if it must exhibit periodic behaviour. Hence, it is natural to ask if diffusion, on any graph, and from any initial configuration, is always eventually periodic (and we urge the reader to pause at this juncture and consider this problem before proceeding further). Duffy, Lidbetter, Messinger and Nowakowski~\citep{diff} raised this precise problem and conjectured, motivated by overwhelming numerical evidence, that diffusion is always eventually periodic with period either 1 or 2; our goal here is to prove this gorgeous conjecture.

A more formal description of diffusion, which is a cellular automaton on a finite graph, is as follows. Let $G$ be an $n$-vertex graph on the vertex set $[n] = \{1, 2, \dots, n\}$. At time $t=0$, each vertex $v \in [n]$ is assigned an initial integer label $w_v(0)$. We then update these labels at discrete time steps according to the following rule: at time $t \ge 0$, for a vertex $v \in V(G)$, if $A_v(t)$ is the number of neighbours $u$ of $v$ with $w_u(t) > w_v(t)$, and $B_v(t)$ is the number of neighbours $u$ of $v$ with $w_u(t) < w_v(t)$, then we set 
\[w_v(t+1) = w_v(t)+A_v(t) -B_v(t).\] For each $t\ge 0$, let $w_G(t)\in \Z^n$ denote the vector $(w_1(t), w_2(t), \dots, w_n(t))$. In this language, the diffusion process on $G$ from the initial configuration $w_G(0) \in \Z^n$ is eventually periodic if the sequence $(w_G(t))_{t\ge 0}$ is eventually periodic. We shall establish the following, thereby settling the aforementioned conjecture due to Duffy, Lidbetter, Messinger and Nowakowski~\citep{diff}.

\begin{theorem}\label{period2}
Diffusion on any graph, and from any initial configuration, is eventually periodic with period either 1 or 2; in other words, for any $n$-vertex graph $G$ and any initial configuration $w_G(0) \in \Z^n$, the sequence $(w_G(t))_{t\ge 0}$ is eventually periodic with period either 1 or 2.
\end{theorem} 

This short note is organised as follows. We prove Theorem~\ref{period2} in Section~\ref{proof}, and we conclude in Section~\ref{concs} with a discussion of some open problems.

\section{Proof of the main result}\label{proof}
Our proof of Theorem~\ref{period2} hinges on the definition of an integer-valued potential function. We shall show that this potential is bounded below, and also that this potential is non-increasing with time; finally, we shall also show that once our potential function stops decreasing (and is consequently constant for the rest of all time), the diffusion process must then attain periodicity with period either 1 or 2. Of course, once we write down the appropriate potential, the rest of the argument is quite straightforward; finding the right definition is hence the crux of the matter.

\begin{proof}[Proof of Theorem~\ref{period2}]
In diffusion on an $n$-vertex graph $G$ from an initial configuration $w_G(0) \in \Z^n$, we define the potential $P(t)$ of the diffusion process at time $t$ by
\[P(t)=\sum_{v = 1}^n w_v(t)w_v(t+1).\]
Let us note two somewhat unexpected features of this potential. First, it is slightly surprising that our potential at a time $t$ depends on the labels of the vertices at both times $t$ and $t+1$. Second, and perhaps more surprisingly, this potential does not appear to take into direct account the structure of the underlying graph, in the sense that the potential merely involves a sum over the vertex set, and completely ignores the edge set!

We first observe that our potential function is bounded below.
\begin{lemma}\label{Pbounded}
For all $t\ge 0$, we have $P(t)\ge -n(n-1)^2/4$.
\end{lemma}
\begin{proof}
This follows immediately from the observation that $|w_v(t+1)-w_v(t)|\le n-1$ for each $v \in [n]$; therefore, for each $v \in [n]$, we have $w_v(t)w_v(t+1) \ge -(n-1)^2/4$, and the claim follows.
\end{proof}

To show that our potential function is non-increasing with time, we shall assign some labels to the edges of $G$ at each time $t \ge 0$. Roughly speaking, at each time $t \ge 0$, we label each edge of $G$ according to the directions in which chips are passed along that edge in the next two steps. More precisely, at a time $t\ge 0$, an edge $uv$ of $G$ with $1 \le u<v \le n$ gets assigned the label $(x_{uv} (t), y_{uv} (t))$ as follows: we set $x_{uv}(t) = \sgn(w_u(t) - w_v(t))$ and $y_{uv}(t) = \sgn(w_u(t+1) - w_v(t+1))$, where $\sgn(m)$ is equal to either $-1$, $0$ or $1$ respectively according to whether $m < 0$, $m=0$ or $m > 0$. We now observe the following.

\begin{lemma}\label{decreasing}
For all $t\ge 0$, we have $P(t+1)\le P(t)$; furthermore, if any edge of $G$ is labelled either $(1,1)$, $(-1,-1)$, $(0,1)$ or $(0,-1)$ at time $t$, then $P(t+1)<P(t)$.
\end{lemma}
\begin{proof}
Observe that 
\[P(t+1)-P(t)=\sum_{v = 1}^n w_v(t+1)(w_v(t+2) -w_v(t)).\]
With the convention that $(x_{uv}(t), y_{uv}(t)) = (0,0)$ whenever $uv$ is not an edge of $G$, we have 
\[w_v(t+2)= w_v(t) + \sum_{u \ne v} \sgn(v-u) (x_{uv}(t) +y_{uv}(t)).\]
Consequently, it follows that
\begin{align*}
P(t+1)-P(t) &=\sum_{v = 1}^n w_v(t+1)\left(\sum_{u \ne v} \sgn(v-u) (x_{uv}(t) +y_{uv}(t))\right)\\
&=\sum_{u<v}(x_{uv}(t)+y_{uv}(t))(w_v(t+1)-w_u(t+1)).
\end{align*}
Consider the contribution $(x_{uv}(t)+y_{uv}(t))(w_v(t+1)-w_u(t+1))$ from a pair of vertices $u, v \in [n]$ with $u < v$ to the above sum. Clearly, this contribution is zero if $x_{uv}(t)+y_{uv}(t) = 0$. Now, suppose that $x_{uv}(t)+y_{uv}(t) \ne 0$; of course, this is only possible when $uv$ is in fact an edge of $G$. If $x_{uv}(t)+y_{uv}(t) > 0$, then $y_{uv}(t)\ge 0$ and this implies that $w_v(t+1) - w_u(t+1) \le 0$, and if $x_{uv}(t)+y_{uv}(t) < 0$, then $y_{uv}(t)\le 0$ and this implies that $w_v(t+1) - w_u(t+1) \ge 0$. Therefore, each term in the above sum is at most zero, and so $P(t+1)\le P(t)$, proving the first claim.
	
Now, if any edge $uv$ is labelled with one of the four labels $(1,1)$, $(-1,-1)$, $(0,1)$ or $(0,-1)$ at time $t$, then we see that the corresponding term $(x_{uv}(t)+y_{uv}(t))(w_v(t+1)-w_u(t+1))$ is negative. For example, if $x_{uv}(t)=0$ and $y_{uv}(t) = 1$, then we have $x_{uv}(t)+y_{uv}(t) = 1$ and $w_v(t+1) - w_u(t+1) < 0$; the three other cases are similarly easy to handle, and this establishes the second claim.
\end{proof}

We may now finish the proof as follows. By Lemma~\ref{Pbounded}, we see that our potential $P(t)$ is bounded below for all $t\ge 0$, and by Lemma~\ref{decreasing}, we see that $P(t)$ is non-increasing with $t$. Since $P(t)$ is integer-valued, there exists some finite time $T$ (depending on our graph $G$ and the initial configuration $w_G(0)$) such that $P(t)$ is constant for all $t \ge T$. It further follows from Lemma~\ref{decreasing} that at each time $t\ge T$, the label of each edge belongs to the set $\{(1,-1),(-1,1),(0,0),(1,0),(-1,0)\}$.

We claim that there exists a time $T' \ge T$ at which the label of each edge belongs to the set $\{(1,-1),(-1,1),(0,0)\}$.
To see this, we first note that if an edge has labels $(i,j)$ and $(k,l)$ at times $t$ and $t+1$, then $j = k$. Furthermore, we also know that an edge cannot be labelled either $(1,1)$, $(-1,-1)$, $(0,1)$ or $(0,-1)$ at any time $t \ge T$. Consequently, we deduce that
\begin{enumerate}
\item if an edge is labelled either $(1,0)$, $(-1,0)$ or $(0,0)$ at some time $t \ge T$, then it must be labelled $(0,0)$ at time $t+1$, and consequently, it must be labelled $(0,0)$ at each time $t' \ge t+1$.
\item if an edge is labelled $(-1,1)$ at some time $t \ge T$, then it must be labelled either $(1,-1)$ or $(1,0)$ at time $t+1$, and
\item if an edge is labelled $(1,-1)$ at some time $t \ge T$, then it must be labelled either $(-1,1)$ or $(-1,0)$ at time $t+1$.
\end{enumerate}
If an edge is labelled either $(1,0)$, $(-1,0)$ or $(0,0)$ at time $T$, then it is labelled $(0,0)$ at each time $t \ge T+1$. If an edge is labelled either $(1,-1)$ or $(-1,1)$ at time $T$, then there are two possibilities: either the label of this edges alternates between $(1,-1)$ and $(-1,1)$ for the rest of all time, or the label of this edge changes to either $(1,0)$ or $(-1,0)$ at some time $t \ge T+1$, and is then labelled $(0,0)$ at each time $t' \ge t+1$. Since $G$ has finitely many edges, it is now clear that there exists a time $T' \ge T$ at which the label of each edge belongs to the set $\{(1,-1),(-1,1),(0,0)\}$.

Finally, note that if the label of each edge belongs to $\{(1,-1),(-1,1),(0,0)\}$ at some time $t$, then we must have $w_G(t) = w_G(t+2)$; indeed, at that time one of two things happens across each edge: either there is no transfer of chips across the edge in question in either of the next two steps, or a chip travels back and forth across the edge in question in the next two steps. Consequently, we have $w_G(t+2) = w_G(t)$ for all $t \ge T'$, proving the result.
\end{proof}

\section{Conclusion}\label{concs}
It is natural to ask if Theorem~\ref{period2} holds under more general conditions. First, we remark that our proof runs essentially as described even when the underlying graph $G$ is allowed to contain parallel edges (so that each vertex fires one chip along \emph{each} edge to each of its neighbours with fewer chips), and when the initial configuration $w_G(0)$ is a vector of real numbers rather than integers. To deal with real-valued labels, one requires a small additional observation, which is that while the potential is no longer integer-valued, it can only attain finitely many distinct values between the lower bound given by Lemma~\ref{Pbounded} and its initial value. Next, while it is easy to see that diffusion on an infinite graph need not be periodic, it would be good to decide whether one can say anything interesting in the case of, say, infinite graphs of bounded degree: for instance, it would be interesting to decide if diffusion on an infinite graph of bounded degree from an initial configuration where the vertex labels are also bounded results in a process where the vertex labels remain bounded for all time.

Duffy, Lidbetter, Messinger and Nowakowski~\citep{diff} raise various other questions about diffusion that are not addressed here, and we conclude by mentioning a problem in a similar vein. Note that the dynamics of diffusion are unchanged if we initially add a fixed number of chips to each vertex. Since we have shown that diffusion is eventually periodic (and consequently bounded), it would be interesting to decide if, for each $n\in \N$, there exists an integer $f(n)\ge 0$ with the property that in diffusion on any $n$-vertex graph where each initial vertex label is at least $f(n)$, all the vertex labels are non-negative at all subsequent times. A star on $n$ vertices shows that $f(n)$, if it exists, must grow at least linearly in $n$; it is conceivable that this is the truth.

\bibliographystyle{amsplain}
\bibliography{diffusion}

\end{document}